\documentclass{article}

\usepackage[T1]{fontenc}

\usepackage{tgtermes}

\usepackage{mathptmx}  
\usepackage{setspace}
\usepackage{bm}
\bmdefine{\bmid}{\mid}
\usepackage[scaled=.92]{helvet}

\usepackage{amsthm,amsmath,amssymb}
\usepackage{mathrsfs}
\usepackage{amsbsy}
\newcommand{\intt}[1]{ [\![  #1 ]\!]}

\usepackage[title,toc,titletoc,header]{appendix}

\usepackage[numbers]{natbib}  %% numbers is required.

\usepackage[colorlinks,citecolor=blue,urlcolor=blue]{hyperref}  %%check
\usepackage{bussproofs}
\usepackage{enumerate}

\usepackage{xcolor}

\providecommand{\keywords}[1]{\textbf{\textit{Keywords: }} #1}

\usepackage{authblk}
\usepackage{hyperref}

\newtheorem{theorem}{Theorem}[section]
\newtheorem{lemma}[theorem]{Lemma}

\theoremstyle{definition}
\newtheorem{definition}[theorem]{Definition}

\newtheorem{proposition}[theorem]{Proposition}
\newtheorem{example}[theorem]{Example}
\theoremstyle{remark}
\newtheorem{remark}[theorem]{Remark}

\newcommand{\ci}{{\bf{i}}}
\newcommand{\ck}{{\bf{k}}}
\newcommand{\cs}{{\bf{s}}}
\newcommand{\CS}{{\sf{S}}}
\newcommand{\CK}{{\sf{K}}}
\newcommand{\CI}{{\sf{I}}}
\newcommand{\LCL}{$LCL$}
\newcommand{\TCL}{\text{\scriptsize CL}}
\begin{document}
%\maketitle
\title{Logic of Combinatory Logic}

\author[1]{{Simona Ka\v sterovi\' c}\thanks{Corresponding author.}}
\author[1,2]{Silvia Ghilezan}
\affil[1]{University of Novi Sad, Faculty of Technical Sciences }
   
\affil[2]{Mathematical Institute SASA, Belgrade, Serbia}
  
\date{{\{simona.k, gsilvia\}@uns.ac.rs}
   }

\maketitle
\begin{abstract}
We develop a classical propositional logic for reasoning about combinatory logic. We define its syntax, axiomatic system and semantics. The syntax and axiomatic system are presented based on classical propositional logic, with typed combinatory terms as basic propositions,  along with the semantics based on applicative structures extended with special elements corresponding to primitive combinators. Both the equational theory of untyped combinatory logic and the proposed axiomatic system are proved to be sound and complete w.r.t. the given semantics. In addition, we prove that combinatory logic is sound and complete w.r.t. the given semantics. 
\end{abstract}

 \section{Introduction}\label{sec: Introduction}
\vspace{0.5em}

Combinatory logic was invented a century ago, in the 1920s, when its basic idea was introduced by Moses Sch\" onfinkel \cite{S1924} (a recent comprehensive overview can be found in \cite{wolfram2021}). The foundations of combinatory logic were established by Haskell Curry \cite{Curry1930} in the 1930s and it has been developing ever since. 
It is of the same expressive power as the $\lambda$-calculus, invented by Alonso Church in the 1930s \cite{C40}, capturing all computable functions nonetheless  without using bound variables. With respect to this, the syntax of combinatory logic is simpler and the obstacles that emerge from bound variables in $\lambda$-calculus are completely avoided in combinatory logic.

Typed combinatory logic 
is a system in which application is controlled by types assigned to terms  \cite{C34}.  Types in logic were firstly introduced informally  by Russell and Whitehead in \cite{russell25}. This type theory was simplified into the simple theory of types in \cite{Ramsey26}. This simple theory of types  was the basis for many modern type systems and for the simply typed lambda calculus \cite{C40}. In \cite{KamareddineLN02}, the authors present the prehistory of type theory, the development of types between Principia Mathematica (\cite{russell25}) and simply typed lambda calculus (\cite{C40}), and the comparison between these type systems.
After a first type system was introduced, now called {\em simple types}, various type systems were proposed, such as intersection types, dependent types, polymorphic types, among others, and found their applications in programming languages, automated theorem provers and proof assistants \cite{BDS13}. Due to a wide range of its applications, combinatory logic, both untyped and typed, has been a topic of interest and the object of many studies, e.g. \cite{Bendowski15,BessaiDDMR14,Bimbo04,Bunder02,
DudderMRU12,DunnM97,Smullyan1985}. 
Developments, some of them very recent,  in a wide range of different scientific fields of computer science, e.g. program synthesis \cite{DudderMRU12}, machine learning, e.g. \cite{LiangJK10}  and artificial intelligence, e.g. \cite{GarretteDBS15}, cognitive representation, e.g. \cite{D04}, natural language, e.g. \cite{SB11}, physics, e.g. \cite{ToffanoD20}, just to mention a few, urge for further research and development  of the theory and reasoning about combinatory logic both typed and untyped.

Various extensions of combinatory logic, both untyped and typed, have been considered in order to obtain formalisms capable to express new features and paradigms, e.g. extending syntax with new constructors such as pairs, records, variants, among others, to enable building compound data structures, better organization of data and  dealing with heterogeneous collections of values \cite{Pierce2002} or extending the syntax with new operators, for example, adding probabilistic operator shifts to a new paradigm, called probabilistic computation \cite{LagoSA14, KasterovicP19}. Herein, the approach we are  interested in combines typed combinatory logic with classical propositional logic in order to obtain the formal system {\em logic of combinatory logic} to reason about typed combinatory terms. 

There are several reasons to be interested in extending combinatory logic in this way. From the perspective of reasoning, the proposed system is a standardized way to syntactically represent and reason about typed combinatory terms.  This is done by considering a reformulation, and slight extension, of the language of classical propositional logic. In turn, well-established methods such as DPLL, resolution method, SAT solvers, SMT solvers can be employed to reason about combinatory logic, which can encompass the development of tool to automatically reason about combinatory terms and programs, due to the Curry-Howard correspondence.  
From the perspective of programming language theory, this is a convenient framework for abstract syntax: the structure of programming languages, disregarding the superficial details of concrete syntax.
From a proof-theoretic perspective, logic of combinatory logic provides a precise correspondence between syntactic and semantic structure: the rules of classical propositional logic give 
an internal language for reasoning about combinatory logic. % 
Another motivation for this work comes from logic-based knowledge representation formalisms, particularly description logic, which is used to represent the conceptual knowledge of an application domain in a structured and formal way, \cite{Baader09}.  Application domains can be very different and so can formal logics for reasoning about them. The Web Ontology Language OWL is a  description logic based language which is nowadays the standard ontology language for the semantic web. In this context, our approach can be seen as a kind of description logic of combinatory logic and the semantics and methodology we develop can be applied to other description logic based languages.

Systems constructed in a similar manner comprise applicative theories by Feferman~\cite{F79}, lambda logic by Beeson~\cite{Beeson04}, logic for reasoning about reversible logic circuits by Axelsen et al.~\cite{AxelsenGK16}. 
Though our approach is similarly motivated, to the best of our knowledge, the system defined in this paper has not been studied before: we define classical propositional logic over simply typed combinatory logic in order to capture inference of type assignment statements. 

\vspace{0.7em}
\subsection*{Contributions}
\vspace{0.5em}

We introduce in this paper a classical propositional logic for reasoning about simply typed combinatory logic, called {\em logic of combinatory logic} and denoted by \LCL. The logic \LCL \, can be explained in two ways:

\begin{itemize}
\item it is a logic obtained by extending the simply typed combinatory logic with classical propositional connectives, and corresponding axioms and rules; 
\vspace{0.5em}
\item it is a logic obtained from classical propositional logic by replacing propositional letters with  type assignment statements $M : \sigma$, i.e. typed combinatory terms, where $M$ is a term of the combinatory logic, and $\sigma$ is a simple type.
\end{itemize}

 We define \LCL\ syntax, axiomatic system and semantics. The axiomatic system (\ref{sec:axiomatization}) has two kinds of axioms: non-logical ones concerned with features of combinatory logic and logical ones concerned with logical connectives. 
  The semantics (\ref{sec:semantics}) is based on the notion of an applicative structure extended with special elements corresponding to the primitive combinators,  inspired by the applicative structures introduced in \cite{MitchellM87,MitchellM91} and \cite{KasterovicG20}.

   The first main result is the {\bf soundness and completeness of the equational theory of untyped combinatory logic} with respect to the proposed semantics, proved in Theorem~\ref{thm:soundness of eq theory} and Theorem~\ref{thm:completeness of eq theory}, respectively. 
   
  The second main result of the paper is the {\bf soundness and completeness of the \LCL\ axiomatization} with respect to the proposed semantics, proved in Theorem~\ref{thm:soundness} and Theorem~\ref{thm:completeness2}, respectively. The completeness theorem is proved by an adaptation of Henkin-style completeness method. We first prove that every consistent set can be extended to a maximal consistent set, which is further used  to introduce the notion of a canonical model. Finally, we prove that every consistent set is satisfiable, and as a consequence we have the completeness theorem: whenever a formula is a semantical consequence of a set of formulas, it is also a deductive consequence of that set. 
  
  In addition, we prove soundness and completeness of combinatory logic in Theorem~\ref{thm:soundness of CL} and Theorem~\ref{thm:completeness of CL}, respectively, which yields the third main result, a {\bf new semantics of combinatory logic}.
  
  Finally, we prove in Theorem~\ref{thm:conservative extension} that \LCL\ is a conservative extension of the simply typed combinatory logic.
  
To the best of our knowledge, the system defined in this paper has not been studied before.

\subsection*{Related work}

The idea to combine different logical systems with the aim to capture reasoning about certain logical structures has been present since Dana Scott's deductive system for computable functions proposed in 1969 and unpublished until 1993 \cite{Scott93}, the unpublished work of John McCarthy, as remarked in \cite{Beeson04}, and Feferman~\cite{F79}.

  An approach similar to ours is presented in \cite{Beeson04}, where  the lambda logic is introduced as the union of first-order logic and $\lambda$-calculus, in order to obtain a more powerful tool for representing the notion of a function.

  Systems of illative combinatory logic, studied in \cite{BarendregtBD93,Czajka11,Czajka13,DekkersBB98},  consist in extending the theory of combinators %or $\lambda$-calculus 
  with additional constants, and corresponding axioms and rules, in order to capture inference. The "Fool's model" in \cite{Meyer91} is a model of combinatory logic in naive set theory which relates untyped combinatory logic and implicational logic, where the models are constructed out of the algebra of implicational formulas.
  
  Despite the similarities in the concept, all the mentioned  approaches consist in extending untyped calculi, whereas we consider the extension of typed combinatory logic with propositional connectives in order to capture the inference of type assignment  statements. On the other hand, in \cite{Scott93} the typed system of combinators, including fixed-point combinator, is extended with logical constants and connectives and the system called Logic of Computable Functions (LCF) intended for reasoning about recursively defined functions is obtained.
  
The approach of \cite{AxelsenGK16} is in  a different sense related to ours. 
A classical propositional logic was introduced  there for reasoning about reversible logic circuits, where reversible logic circuits are represented as propositions in a variant of classical propositional logic extended with ordered multiplicative conjunction. In a broader context, this work is related to ours since they share the same philosophy of extending a basic logical system, namely, reversible logic circuits and combinatory logic, herein, with classical propositional logic in order to formally reason about the basic logical system.

Different models for combinatory logic have been proposed, and soundness and completeness of both equational calculi and type assignment systems with respect to those models  have been  studied \cite{Bimbo04, Bimbo12, Barendregt71, Barendregt20, AF20}. However, as far as we know, the novel techniques presented in this paper were not considered in the previous approaches.

  \subsection*{Organisation of the paper}

The rest of the paper is organised as follows. Section~\ref{sec:preliminaries} reviews some basic notions of simply typed combinatory logic, its syntax, reductions and equational theory. Section~\ref{sec:LCL}  introduces the logic of combinatory logic: its syntax, axiomatic system and semantics.  Section~\ref{sec:equational soundness and completeness}  proves the soundness and completeness  of the equational theory of untyped combinatory logic w.r.t. the proposed semantics.  Section~\ref{sec:semantical soundness and completeness} proves the soundness and completeness  of the given axiomatic system w.r.t. the given semantics 
and proves that combinatory logic is 
sound and complete with respect to the proposed semantics. Section~\ref{sec:conclusion}  concludes and contains discussion of further work.

\section{Simply Typed Combinatory Logic}\label{sec:preliminaries}
\vspace{0.5em}

In this section we revisit some basic notions of simply typed combinatory logic $CL_\rightarrow$~\cite{B85,Bimbo12, hindley_seldin_2008}. 
\vspace{0.5em}
\subsection{Syntax}
\vspace{0.5em}
The alphabet of untyped combinatory logic, $CL$, consists of a countable set of term variables $V$, term constants (or constants for short) and a binary operation called {\em application}, denoted by juxtaposition.   $CL$-terms are generated by the following syntax:
\begin{align*} 
\displaystyle M, N  ::= & \; x \;  {\mid} \; \CS \;  {\mid} \; \CK \; {\mid } \; \CI \;  {\mid } \; M N
\end{align*}
\noindent
where $x$ is a term variable. The set of all terms is denoted by $CL$. We include in the language the constants $\CS, \CK, \CI$; however there are other constants that can be included such as $\sf{B}, \sf{W}, \sf{C}, \ldots $, see~\cite{Bimbo12}.   The constants $\CS, \CK, \CI$ are called {\em primitive combinators}. The set of all  $CL$-terms is ranged over by capital letters $M, N, \ldots, M_1, \ldots$. By $FV(M)$ we will denote the set of variables that occur in $M$. By $M\{N / x \}$ we will denote the result of substitution of the term $N$ for variable $x$ in the term $M$.

Combinators are characterized by the number of terms they are applied to (number of arguments they can have) and the effect they have when they are applied to other terms. The following axioms characterize the behavior of the primitive combinators $\CS, \CK, \CI$, they show the effect that  combinators  $\CS, \CK, \CI$ have  when applied to other terms:

\begin{align*}
\CS M N L & \rightarrow (ML)(NL) \\
 \CK M N & \rightarrow M  \\ 
  \CI M & \rightarrow M 
\end{align*}

Terms $\CS M N L, \CK M N, \CI M $ are called {\em redexes}. A redex can be a subterm of some term. In order to reduce  a term which contains a redex, a one-step reduction $\triangleright_1$ is introduced.
\begin{definition}
 If $N$ is a redex such that $N \rightarrow N'$ and $N$ is a subterm of a term $M$, then $M \triangleright_1 M'$, where $M'$ is obtained from $M$ by replacing the subterm $N$ with $N'$.
\end{definition}
The reflexive and transitive closure of one-step reduction $\triangleright_1$ is called {\em weak reduction} and it  is denoted by $\triangleright_w$. The reflexive, transitive and symmetric closure of the reduction $\triangleright_1$ is called {\em weak equality} and  denoted by $=_w$. Weak equality can be inductively characterized as follows: \begin{enumerate}
\item If $M \triangleright_1 N$, then $M =_w N$.
\item If $M$ is a $CL$-term, then $M =_w M$.
\item If $M =_w N$, then $N =_w M$.
\item If $M =_w N$ and $N =_w L$, then $M =_w L$.
\end{enumerate}

{\em Extensional weak equality}, denoted by $=_{w, \eta}$, is obtained from  $\triangleright_w$ as follows.

\begin{definition}
$CL$-terms $Q$ and $Q\{M_1 / N_1, \ldots, M_m / N_m \}$ are extensionally weakly equal, denoted by $Q =_{w, \eta} Q\{M_1 / N_1, \ldots, M_m / N_m \}$ if and only if  for every $i \in \{1, \ldots, m \}$, there exists $n_i$ such that for a series of distinct variables $x_{i, 1}, x_{i,2}, \ldots, x_{i, n_i}$, that do not occur neither in  $M_i$ nor  $N_i$, there is  $P_i$ with $M_i x_{i, 1}, x_{i, 2}, \ldots x_{i, n_i} \triangleright_w P_i$  and  $N_i x_{i, 1} x_{i, 2} \ldots x_{i, n_i} \triangleright_w P_i$.
\end{definition}

 \vspace{0.5em}
 \subsection{Equational theory}\label{sec:eqtheory}
\vspace{0.5em}

The axioms and rules in Fig.~\ref{fig:eq-calc} define an equational theory $\mathcal{EQ}$. We say that $CL$-terms $M$ and $N$ are equal, $M = N$, if  $M = N$ can be derived from the set of axioms and rules presented in Fig.~\ref{fig:eq-calc}. The equational theory $\mathcal{EQ}$ determines a new equivalence relation on the set of all $CL$-terms. The introduced equivalence relations, namely, the equivalence relation generated by the equational theory $\mathcal{EQ}$ and the weak equality $=_w$ coincide on the set of all $CL$-terms.

\begin{figure}[!ht]\centering
\begin{tabular}{cccc} 
\bigskip
\AxiomC{$ M = M$ \quad $(id)$ }
\DisplayProof 
 & 
 \AxiomC{$\CS M N L = (ML)(NL)$ \quad $(S)$}
\DisplayProof 
 &
 \AxiomC{$ \CK M N  = M$ \quad $(K)$}
\DisplayProof  \bigskip \\ \bigskip

 \AxiomC{$\CI M = M$ \quad $(I)$}
 \DisplayProof  
&
\AxiomC{$M = N$}
\RightLabel{(sym)}
\UnaryInfC{$N = M$}
\DisplayProof &
\AxiomC{$M = N$}
\AxiomC{$N = L$}
\RightLabel{(trans)}
\BinaryInfC{$M = L$}
\DisplayProof \bigskip \\ \bigskip
 \AxiomC{$M = N$}
 \RightLabel{(app-l)}
 \UnaryInfC{$MP = NP$}
 \DisplayProof
 &
 \AxiomC{$M = N$}
 \RightLabel{(app-r)}
 \UnaryInfC{$PM = PN$}
 \DisplayProof \bigskip

\end{tabular}\caption{Equational theory $\mathcal{EQ}$}\label{fig:eq-calc}
\end{figure}

\begin{proposition} An equation $M = N$ is provable in $\mathcal{EQ}$ if and only if $M =_w N$.
\end{proposition}
\begin{proof}
See~\cite{Bimbo12}.
\end{proof}

\vspace{0.5em}

In order to obtain an extensional version of the equational theory $\mathcal{EQ}$, we  extended equational theory $\mathcal{EQ}$ with the following rule:
\[\AxiomC{$Mx = Nx$}
\AxiomC{$x \not\in FV(M) \cup FV(N)$}
\RightLabel{(ext)}
\BinaryInfC{$M = N$}
\DisplayProof \]

The resulting theory is  denoted by $\mathcal{EQ}^\eta$. The equational theory $\mathcal{EQ}^\eta$ generates an equivalence relation which coincides with extensional weak equality.

\begin{proposition} An equation $M = N$ is provable in $\mathcal{EQ}^\eta$ if and only if $M =_{w, \eta} N$.
\end{proposition}
\begin{proof}
See \cite{Bimbo12}.
\end{proof}

\vspace{0.5cm}

The equational theory $\mathcal{EQ}$ and its connection to weak equality $=_w$, as well as an inequational theory, which we do not consider in the paper, and its connection to weak reduction are elaborated in detail in \cite{Bimbo12}.

\subsection{Types}
We are interested only in typed $CL$-terms, more precisely, we are interested in simply typed $CL$-terms. Let $V_{\sf Type} = \{a, b, c, . . ., a_1, \ldots \}$ be a countable set of  type variables.  {\em Simple types} are generated by the following syntax:

\[ \sigma, \tau ::= a \mid \sigma \rightarrow \tau \]
\noindent
where $a \in V_{\sf Type}$. The set of all simple types is denoted by ${\sf Types}$ and is ranged over by $\sigma, \tau, \ldots, \sigma_1, \ldots$.

\begin{definition}\label{def:statement...} $ $

\begin{enumerate}
\item[(i)] {\em A (type assignment)  statement} is of the form $M : \sigma$, where $M\in CL$ and  $\sigma \in {\sf Types}$. The term $M$ is the {\em subject} and the type $\sigma$ is the {\em predicate} of the statement.
\item[(ii)] A  statement with a term variable as subject, $x : \sigma$, is called {\em a declaration}.
\item[(iii)]  {\em A basis} ({\em context}) is a set of declarations with distinct term variables as subjects.
\item[(iv)] Let $\Gamma = \{x_1 : \sigma_1, \ldots, x_n : \sigma_n \}$ be a basis. The set $dom(\Gamma) = \{ x_1, \ldots, x_n \}$ is the {\em domain} of the basis $\Gamma$.
\item[(v)] Let $\Gamma = \{x_1 : \sigma_1, \ldots, x_n : \sigma_n \}$ be a basis. We define the set $| \Gamma | = \{ \sigma_1, \ldots, \sigma_n \}$.
\end{enumerate}
\end{definition}

The type assignment system for {\em simply typed combinatory logic}, denoted by $CL_\rightarrow$,  is presented in Fig.~\ref{fig: type assignment}. 
 If the typing judgment $\Gamma \vdash_{\TCL} M : \sigma$ can be derived by the rules of Fig.~\ref{fig: type assignment}, then we say that a statement $M:\sigma$ is derivable from a basis $\Gamma$. By $CL_\rightarrow$ we will also denote the set  of all statements $M:\sigma$ which are typable from some basis $\Gamma$, i.e. \[CL_\rightarrow =  \{ M : \sigma \mid \exists \Gamma. \ \Gamma \vdash_{\TCL} M : \sigma \}. \]

\begin{figure}[!t]
\centering
\begin{tabular}{cc}

\multicolumn{2}{c}{   \AxiomC{$x : \sigma \in \Gamma$}
    \RightLabel{(axiom $\in$)}
    \UnaryInfC{$\Gamma \vdash_{\TCL} x :\sigma$} 
    \DisplayProof  \bigskip} \\ 
     \multicolumn{2}{c}{ \bigskip\AxiomC{}
    \RightLabel{(axiom $S$)}
    \UnaryInfC{$\Gamma \vdash_{\TCL} \CS: (\sigma \rightarrow (\rho \rightarrow \tau)) \rightarrow (\sigma \rightarrow \rho)  \rightarrow (\sigma \rightarrow \tau)$} 
    \DisplayProof} \bigskip  \\    
      \AxiomC{}
    \RightLabel{(axiom $K$)}
    \UnaryInfC{$\Gamma \vdash_{\TCL} \CK : \sigma \rightarrow ( \tau \rightarrow \sigma)$} 
    \DisplayProof 
    &
      \AxiomC{}
    \RightLabel{(axiom $I$)}
    \UnaryInfC{$\Gamma \vdash_{\TCL} \CI : \sigma \rightarrow \sigma$} 
    \DisplayProof \bigskip \\ 
   \multicolumn{2}{c}{ \bigskip 
     \AxiomC{$\Gamma \vdash_{\TCL} M: \sigma \rightarrow \tau$}
     \AxiomC{$\Gamma \vdash_{\TCL} N : \sigma$}
    \RightLabel{($\rightarrow$ elim)}
    \BinaryInfC{$\Gamma \vdash_{\TCL} MN : \tau$} 
    \DisplayProof\bigskip  } 
\end{tabular}
\caption{Type Assignment System $CL_\rightarrow$}\label{fig: type assignment}
\end{figure}

We end this section with some properties of the type assignment system $CL_{\rightarrow}$.

\begin{proposition}\label{lemma: weakening lemma} If  $ \Gamma \vdash_{\TCL} M : \sigma $ and $ \Gamma \subseteq \Gamma' $ then $ \Gamma' \vdash_{\TCL} M : \sigma$.
\end{proposition}
\begin{proof}
By induction on the length of derivation $\Gamma \vdash_{\TCL} M : \sigma$.
\end{proof}

\vspace{0.5cm}

For a set $X$ of term variables, we write $\Gamma \upharpoonright X = \{ x : \sigma \in \Gamma \mid x \in X \}$.

\begin{proposition}\label{lemma:free variable lemma}\begin{enumerate}
\item If $\Gamma \vdash_{\TCL} M : \sigma$, then $FV(M) \subseteq dom (\Gamma).$
\item If $\Gamma \vdash_{\TCL} M :\sigma$, then $\Gamma \upharpoonright FV(M) \vdash_{\TCL} M  :\sigma $.
\end{enumerate}
\end{proposition}
\begin{proof}
$1), 2)$ By induction on the length of  derivation $\Gamma \vdash_{\TCL} M : \sigma$.
\end{proof}

\vspace{0.5em}
\section{Logic of Combinatory Logic: its syntax and semantics}\label{sec:LCL}
\vspace{0.5em}

In this section we introduce the {\em logic of combinatory logic}, denoted by \LCL, as a logical system obtained by expanding simply typed combinatory logic with classical logical connectives of negation and implication. We define its syntax, axiomatic system and semantics.

\vspace{0.5em}
\subsection{Syntax \LCL}
\vspace{0.5em}

The system \LCL \ is the classical propositional extension of simply typed combinatory logic. %{\color{red} Let ${\bf TS}$ be the set of the statements  $M:\sigma$ which are typable from some basis $\Gamma$, i.e. \[ {\bf TS} = \{ M : \sigma \mid \exists \Gamma. \ \Gamma \vdash_{\TCL} M : \sigma \}. \]} 
The language of the logic \LCL \ is generated by the following grammar
\[ \alpha,\beta := M : \sigma \mid \neg \alpha \mid \alpha \Rightarrow \beta \]
where  $M : \sigma \in CL_\rightarrow$. 
 Other classical propositional connectives $\wedge, \vee, \Leftrightarrow$ are defined in the standard way. The formulas of the logic \LCL \ will be denoted by $\alpha, \beta, \ldots$ possibly with subscripts.   We denote $\alpha \wedge \neg \alpha$ by $\bot$.

The logic we propose can be seen as a first step towards formalization of meta-language of simply typed combinatory logic. For example, in simply typed combinatory logic we have that {\em if  we can derive $M : \sigma \rightarrow \tau$ and $N : \sigma$  from a basis $\Gamma$, then we can derive $MN : \tau$  from $\Gamma$}, and in \LCL \ this can be described by the formula \[(M : \sigma \rightarrow \tau \wedge N : \sigma)  \Rightarrow MN : \tau.\] 

\vspace{0.5em}
\subsection{Axiomatization  \LCL }\label{sec:axiomatization}
\vspace{0.5em}

Hereinafter, we present the axiomatic system for the logic of combinatory logic \LCL. The axiomatic system of \LCL \ is obtained by combining the axiomatic system for classical propositional logic and type assignment system for simply typed combinatory logic. The {\em axiomatic system} of \LCL\  consists of axiom schemes  and one inference rule presented in Fig.~\ref{figure:axiomatization}, where 
\begin{itemize}
\item axioms (Ax $1$), (Ax $2$) and (Ax $3$) are non-logical axioms corresponding to primitive combinators
\item axioms (Ax $4$), (Ax $5$) correspond to rules of simply typed combinatory logic
\item axioms (Ax $6$), (Ax $7$) and (Ax $8$) are logical axioms of classical propositional logic;
\item the rule (MP) is Modus Ponens.
\end{itemize}

\begin{center}
\begin{figure}[!t]
\centering
Axiom schemes:

\bigskip

\begin{tabular}{c l}
(Ax $1$) & $\CS : (\sigma \rightarrow (\tau \rightarrow \rho)) \rightarrow ((\sigma \rightarrow \tau) \rightarrow (\sigma \rightarrow\rho))$\\
(Ax $2$) & $\CK : \sigma \rightarrow (\tau \rightarrow \sigma)$\\
(Ax $3$) & $\CI : \sigma \rightarrow \sigma$\\
(Ax $4$) & $(M :\sigma \rightarrow \tau) \Rightarrow ((N : \sigma) \Rightarrow (MN : \tau))$, \quad  $M :\sigma \rightarrow \tau, N : \sigma, MN : \tau \in {CL_\rightarrow}  $  \\
(Ax $5$) & $M : \sigma \Rightarrow N : \sigma$, if $M =_{w, \eta} N$,  $M : \sigma, N : \sigma \in {CL_\rightarrow}$ \\% is provable in $\mathcal{EQ}^\eta$ \\
(Ax $6$) & $\alpha \Rightarrow (\beta \Rightarrow \alpha)$\\
(Ax $7$) & $(\alpha \Rightarrow (\beta \Rightarrow \gamma)) \Rightarrow ((\alpha \Rightarrow \beta) \Rightarrow (\alpha \Rightarrow \gamma))$\\
(Ax $8$) &  $(\neg  \alpha \Rightarrow \neg \beta) \Rightarrow ((\neg  \alpha \Rightarrow \beta) \Rightarrow  \alpha)$
\end{tabular}

\bigskip

Inference rule:

\bigskip

\begin{tabular}{c}
\AxiomC{$\alpha \Rightarrow \beta$}
\AxiomC{$\alpha$}
\RightLabel{(MP)}
\BinaryInfC{$\beta$}
\DisplayProof\bigskip 
\end{tabular}
\caption{Axiom schemes and inference rules for \LCL}
\label{figure:axiomatization}
\end{figure}
\end{center}

 Please observe that for every axiom (rule) of type assignment system for $CL_{\rightarrow}$ (Fig.~\ref{fig: type assignment}), except for (axiom $\in$), there is an axiom in Fig.~\ref{figure:axiomatization} that corresponds to that axiom (rule).

\begin{definition}\label{def:proof} Let $T$ be a set of \LCL-formulas and $\alpha$ a formula. A formula $\alpha$ can be derived from $T$, denoted by $T \vdash \alpha$, if there exists a sequence of formulas $\alpha_0, \alpha_1  \ldots, \alpha_n$ such that $\alpha_n$ is the formula $\alpha$ and for every $i \in \{0, 1, \ldots, n\}$, $\alpha_i$ is an axiom instance, or $\alpha_i \in T$ or $\alpha_i$ is a formula which can be derived by the inference rule (MP) applied to some previous members of the sequence.
\end{definition}

If $\emptyset \vdash \alpha$, then the formula $\alpha$ is called  a {\em theorem} and we write $\vdash \alpha$.

\begin{example} We will prove that $M : \sigma \rightarrow \tau, N : \sigma \vdash MN : \tau$ holds,  whenever $M : \sigma \rightarrow \tau$, $N : \sigma $ and $MN : \tau$ are \LCL-formulas, i.e. $M : \sigma \rightarrow \tau, N : \sigma, MN : \tau \in {CL_\rightarrow}$. We write the proof as the sequence of formulas, starting with the formulas $M : \sigma \rightarrow \tau$ and $N : \sigma$ which are  hypothesis.
 \begin{align*}
 1. \quad & {M  : \sigma \rightarrow \tau}, \text{ hypothesis}, \\
 2. \quad & {N : \sigma}, \text{ hypothesis}, \\
 3. \quad & (M: \sigma \rightarrow \tau) \Rightarrow ((N : \sigma) \Rightarrow (MN : \tau)), % \\
% & \ \qquad\qquad\qquad\qquad\qquad\qquad\qquad
\text{ an instance of (Ax 4)},\\
 4. \quad & (N : \sigma) \Rightarrow (MN : \tau), \text{ from }  (1) \text{ and }  (3) \text{ by rule (MP)},\\
 5. \quad & MN : \tau, \text{ from } (2) \text{ and } (4) \text{ by rule (MP)}.
 \end{align*}
 \end{example}
 
Please recall that we have introduced the logic \LCL \ as an extension of simply typed combinatory logic. As we have already observed, for every rule of Fig.~\ref{fig: type assignment}, except the first one, (axiom $\in$), there is an axiom in Fig.~\ref{figure:axiomatization}. On the other hand, (axiom $\in$): if $x:\sigma \in \Gamma$, then  $\Gamma \vdash_{\TCL} x : \sigma$, is described by  Definition~\ref{def:proof}. As a consequence, we have the following proposition.

\begin{proposition}\label{proposition:from cl to extension}
If $\Gamma \vdash_{\TCL} M : \sigma$, then $\Gamma \vdash M : \sigma$.
 \end{proposition}
 \begin{proof}
 By induction on the derivation of $\Gamma \vdash_{\TCL} M : \sigma$ in $CL_{\rightarrow}$.
 \end{proof}
 \vspace{0.5cm}

\begin{definition}\label{def:consistent set} A set of \LCL-formulas $T$ is consistent if there exists at least a formula $\alpha$ which is not  derivable %deducible 
from $T$.  Otherwise, $T$ is inconsistent.
\end{definition}

Alternatively, we can say that $T$ is inconsistent if and only if $T \vdash \bot$. We will write $T, \alpha \vdash \beta$, instead of $T \cup \{\alpha \} \vdash \beta$. 

\begin{theorem}[Deduction theorem]\label{thm:deduction} Let $T$ be a set of formulas and $\alpha, \beta$ formulas of \LCL. If $T, \alpha \vdash \beta$, then $T \vdash \alpha \Rightarrow \beta$.
\end{theorem}
\begin{proof}
We prove the statement by induction on the length  of the proof of $T, \alpha \vdash \beta$. If the length of a proof is equal to $1$, then $\beta$ is either an axiom or $\beta \in T \cup \{ \alpha\}$. 

If $\beta$ is an axiom, then the proof proceeds as follows:
\begin{align*}
1. \quad & T \vdash \beta, \text{ since } \beta \text{ is an axiom} \\
2.  \quad & T \vdash \beta \Rightarrow (\alpha \Rightarrow \beta),  \text{ from } \text{(Ax $6$)}\\
3. \quad & T \vdash \alpha \Rightarrow \beta, \text{ from } (1) \text{ and } (2) \text{ by rule (MP)}.
\end{align*} 

If $\beta \in T$, then we have 
\begin{align*}
1. \quad &  T \vdash \beta, \text{ since } \beta \in T \\
2. \quad & T \vdash \beta \Rightarrow (\alpha \Rightarrow \beta), \text{ from } \text{(Ax $6$)}\\
3. \quad & T \vdash \alpha \Rightarrow \beta, \text{ from } (1) \text{ and } (2) \text{ by rule (MP)}. 
\end{align*}

If $\alpha = \beta$ ($\beta \in\{\alpha \}$), then $T \vdash \alpha \Rightarrow \alpha$ holds because $\alpha \Rightarrow \alpha$ is a theorem  in \LCL. 

Suppose that the length of the proof is $k > 1$. If the formula $\beta$ is an axiom or belongs to the set $T \cup\{\alpha \}$ then the proof proceeds as above. Let the formula $\beta$ be derived from $T \cup \{ \alpha\}$ by an application of  inference rule (MP). 
By the induction hypothesis the statement holds for any proof whose length is less than $k$. 

If $\beta$ is obtained by an application of rule (MP) on formulas $\alpha_1$ and $\alpha_1 \Rightarrow \beta$ such that $T, \alpha \vdash \alpha_1$ and $T, \alpha \vdash \alpha_1 \Rightarrow \beta$, then by the induction hypothesis we have 
 $ T \vdash \alpha \Rightarrow \alpha_1$ and $ T \vdash \alpha \Rightarrow (\alpha_1 \Rightarrow \beta)$. The proof proceeds as follows:
\begin{align*}
1. & \quad  T \vdash \alpha \Rightarrow \alpha_1, \text{ by the induction hypothesis}  \\
2. & \quad   T \vdash \alpha \Rightarrow (\alpha_1 \Rightarrow \beta),  \text{ by the induction hypothesis} \\
3. &  \quad T \vdash (\alpha \Rightarrow (\alpha_1 \Rightarrow \beta)) \Rightarrow ((\alpha\Rightarrow \alpha_1) \Rightarrow (\alpha \Rightarrow \beta)),  \text{ from (Ax $7$)} \\
4. & \quad  T \vdash (\alpha\Rightarrow \alpha_1) \Rightarrow (\alpha \Rightarrow \beta), \text{ from } (2) \text{ and } (3)  \text{ by rule (MP)}  \\
5. & \quad T \vdash \alpha \Rightarrow \beta, \text{ from } (1) \text{ and } (4) \text{ by rule (MP)}.
\end{align*}

This concludes the proof.
\end{proof}

Deduction theorem will play an important role in the proof of  completeness of the axiomatization of \LCL.

\vspace{0.5em}
\subsection{Semantics \LCL}\label{sec:semantics}
\vspace{0.5em}

The first models of pure untyped combinatory logic were introduced by Dana Scott in 1969, \cite{Scott69a, Scott69b}. Since then, different models have been proposed both for untyped and typed combinatory logic. One of the first models that have been introduced was constructed from the language of combinatory logic and is called {\em a term model} \cite{Barendregt71, Barendregt20}. 
 Algebraic and set theoretical - relational and operational - semantics for dual and symmetric combinatory calculi are  introduced  in \cite{Bimbo04}. Further, operational models for untyped combinatory logic and models for typed combinatory logic, combinatory algebras and relational models, have been presented in \cite{Bimbo12}. A wide variety of models can be found in the literature \cite{B85, BDS13, AF20}. %% dodati jos neke

The most usual approach  used to define models for typed combinatory logic is to employ models of untyped  combinatory logic and to interpret types as subsets of the untyped model. However, different approaches have emerged in the past decades.  In \cite{MitchellM87}, a Kripke-style semantics for simply typed $\lambda$-calculus was proposed, and an interpretation of a typing sequent was defined  considering the term and its type at the same time. In \cite{KasterovicG20}, this approach was combined with the general concept of a model of untyped $\lambda$-calculus and  a Kripke-style semantics for full simply typed $\lambda$-calculus was introduced, where the meaning of a term is defined without considering its type. Kripke-style models for full simply typed combinatory  logic have also been introduced in \cite{KasterovicG20} and the key point in some proofs was the translation of $\lambda$-terms into terms of combinatory logic. The semantics we propose is inspired by the semantics introduced in \cite{MitchellM87} and \cite{KasterovicG20}.

In this section we define a semantics for the logic of combinatory logic, \LCL. First, we introduce the notion of an applicative structure for \LCL.

\begin{definition}\label{def:applicative structure}
An applicative structure for \LCL \ is a tuple $ \mathcal{M} = \langle D, \{A^\sigma \}_\sigma ,\cdot, \cs,  \ck, \ci \rangle$ where
\begin{itemize}
\item $D$ is a non-empty set, called {\em domain};
\item $\{A^\sigma \}_\sigma$ is a family of sets indexed by $\sigma$ such that $A^\sigma \subseteq D$ for all $\sigma$, $\{A^\sigma \}_\sigma$ is an abbreviation for $\{A^\sigma\}_{\sigma \in {\sf Types}}$;
\item $\cdot$ is a binary operation on $D$ ($\cdot : D \times D \rightarrow D$), which is extensional: for $d_1, d_2 \in D$, $d_1 = d_2$ whenever $(\forall e \in D) (d_1 \cdot e = d_2 \cdot e)$, and it holds that $\cdot : A^{\sigma \rightarrow \tau} \times A^{\sigma} \rightarrow A^\tau$ for every $\sigma, \tau \in {\sf Types}$;
\item $\cs$ is an element of the domain $D$, such that for every $\sigma, \tau, \rho \in {\sf Types}$, \begin{equation}
\cs \in A^{(\sigma \rightarrow (\tau \rightarrow \rho)) \rightarrow ((\sigma \rightarrow \tau) \rightarrow (\sigma \rightarrow \rho))}\label{eq:cs 1}
\end{equation}  and for every $d, e, f \in D$, \begin{equation}
((\cs \cdot d) \cdot e ) \cdot f = (d \cdot f)\cdot (e \cdot f)\label{eq:cs 2}
\end{equation}
\item $\ck$ is an element of the domain $D$ such that for every $\sigma, \tau \in{\sf Types}$ \begin{equation}
\ck \in A^{\sigma \rightarrow (\tau \rightarrow \sigma)}\label{eq:ck 1}
\end{equation} and for every $d, e \in D$, \begin{equation}
(\ck \cdot d) \cdot e = d \label{eq:ck 2}
\end{equation}
\item  $\ci$ is an element of the domain $D$ such that for every $\sigma \in {\sf Types}$ \begin{equation}
\ci \in A^{\sigma \rightarrow \sigma}
\label{eq:ci 1}
\end{equation} and for every $d \in D$, \begin{equation}
\ci \cdot d = d\label{eq:ci 2}
\end{equation}

\end{itemize}
\end{definition}

Elements $\cs, \ck, \ci$ are unique due to the extensionality of the operation $\cdot$.

In order to interpret $CL$-terms in an applicative structure we have to provide a valuation of term variables to that structure, hence we introduce a notion of an environment for an applicative structure. 
   
\begin{definition}\label{def:environment} Let $\mathcal{M}$ be an applicative structure. An environment $\rho$ for $\mathcal{M}$ is a  map from the set of term variables to the domain of the applicative structure $\mathcal{M}$, $\rho : V \rightarrow D$.
\end{definition}

A similar approach was used in \cite{KasterovicG20,MitchellM91}, where the notions of applicative structure and environment was used to define a Kripke-style semantics for typed $\lambda$-calculus. 

If $\rho$ is an environment for an applicative structure $\mathcal{M}$ and $d$ is an element from the domain of $\mathcal{M}$, then $\rho(x := d)$  denotes the environment for $\mathcal{M}$ for which

\[\rho(x := d)(y) = \begin{cases}
d, & y = x\\
\rho(y), & y \neq x
\end{cases} \] 
\begin{definition}\label{def:model} An \LCL-model is a tuple $\mathcal{M}_\rho = \langle \mathcal{M}, \rho \rangle$, where $\mathcal{M}$ is an applicative structure and $\rho$ is an environment for $\mathcal{M}$.
\end{definition}

An environment $\rho$ gives an interpretation (meaning) of  term variables in an applicative structure. We now  define the meaning of a term. In order to do that, we extend the environment (valuation of term variables) to the interpretation of all $CL$-terms.

\begin{definition}\label{def:meaning of term} Let $\mathcal{M}$ be an applicative structure and $\rho$ and environment for $\mathcal{M}$. We define the {\em interpretation (meaning) of a term} $M$ in the environment $\rho$, denoted by $\intt{M}_\rho$, inductively as follows:
\begin{itemize}
\item $\intt{x}_\rho = \rho(x)$;
\item $\intt{\CS}_\rho = \cs$;
\item $\intt{\CK}_\rho = \ck$;
\item $\intt{\CI}_\rho = \ci$;
\item $\intt{MN}_\rho = \intt{M}_\rho \cdot \intt{N}_\rho$.
\end{itemize}
\end{definition}

The interpretation map $\intt{\cdot }_\rho$ is well defined, that is the meaning of every  $CL$-term in an environment $\rho$ is defined, due to the existence of elements $\cs, \ck, \ci$ in every %an 
applicative  structure. 

The meaning of the term does not depend on the variables that do not occur in the term, as it is shown in the next lemma.

\begin{lemma}\label{lemma:diff env same meaning}
Let $\mathcal{M}$ be an applicative structure, $\rho_1$ and $\rho_2$ two environments for $\mathcal{M}$ and $M$ a $CL$-term. If $\rho_1 (x) = \rho_2 (x)$ for every $x$ that occurs in $M$, then $\intt{M}_{\rho_1} = \intt{M}_{\rho_2}$.
\end{lemma}
\begin{proof}
By induction on the structure of term $M$.
\end{proof}

The following lemma states an important property of the meaning of term obtained by substitution, which will be useful in the sequel.

\begin{lemma}[Substitution lemma]\label{lemma:substitution lemma}
Let  $M, N$ be $CL$-terms and $\rho$ an environment.
Then,
\[\intt{M\{N / x \}}_\rho  = \intt{M}_{\rho(x := \intt{N}{_{\rho}})}\]
\end{lemma}
\begin{proof}
By induction on the structure of term $M$.
\end{proof}

\begin{definition}\label{def:satisfiability} The {\em satisfiability of a formula in a model} $\mathcal{M}_\rho$ is defined as follows:
\begin{itemize}
\item $\mathcal{M}_\rho\models M : \sigma$ if and only if $\intt{M}_\rho \in A^\sigma$;
\item $\mathcal{M}_\rho \models \alpha \wedge \beta $ if and only if $\mathcal{M}_\rho \models \alpha$ and $\mathcal{M}_\rho \models  \beta$;
\item $\mathcal{M}_\rho \models  \neg \alpha$ if and only if it is not true that $\mathcal{M}_\rho \models  \alpha$.
\end{itemize}
\end{definition}

If $\mathcal{M}_\rho \models \alpha$, we say that $\mathcal{M}_\rho$ is  a model of $\alpha$. We  introduce the notion of satisfiability of a set of formulas in a model, and the notion of semantical consequence in the usual way.

\begin{definition}\label{def:semantical consequence}
\begin{enumerate}
\item A set $T$ of \LCL-formulas is satisfied in a model $\mathcal{M}_\rho$, denoted by $\mathcal{M}_\rho \models T$, if and only if every formula $\alpha$ from the set $T$ is satisfied in $\mathcal{M}_\rho$, i.e. $\mathcal{M}_\rho \models \alpha$ for every $\alpha \in T$.
\item A formula $\alpha$ is a semantical consequence of a set $T$, denoted by $T \models \alpha$, if and only if every model of $T$ is a model of $\alpha$, i.e. $\mathcal{M}_\rho \models T$ implies $\mathcal{M}_\rho \models \alpha$.
\end{enumerate}
\end{definition}

An interpretation of a classical propositional formula in a valuation is usually presented as a line (valuation) in the truth table, where atomic formulas in the truth table are propositional letters which appear in the interpreted formula. An environment $\rho$  also defines a line (valuation) in the truth table for \LCL-formula $\alpha$, in which atomic formulas are statements $M:\sigma$, which appear in the formula $\alpha$.

\section{Soundness and completeness of the equational theory}\label{sec:equational soundness and completeness}
\vspace{0.5em}

In this section we prove the soundness and completeness of the equational theory $\mathcal{EQ}^\eta$ (\ref{sec:eqtheory}) w.r.t. the semantics of \LCL\, introduced in %Section~
\ref{sec:semantics}.

In %Section~
\ref{sec:semantics}, we have  introduced models for \LCL, the propositional extension of simply typed combinatory logic, such that the interpretation of every $CL$-term is defined in a model. Therefore,
 the problem of equational soundness and completeness arises. The soundness of the equational theory $\mathcal{EQ}^\eta$ follows from the definition of \LCL-model.

\subsection{Soundness}
\begin{theorem}[Soundness of $\mathcal{EQ}^\eta$]\label{thm:soundness of eq theory} If $M = N$ in provable in $\mathcal{EQ}^\eta$, then $\intt{M}_\rho = \intt{N}_\rho$ for any \LCL-model $\mathcal{M_\rho} = \langle \mathcal{M}, \rho \rangle$.
\end{theorem}

\begin{proof}
We prove the theorem by induction on the length of the proof of $M= N$. 

If $M = N$ is obtained from axiom $(id)$, then terms $M$ and $N$ are identical, so we have $\intt{M}_\rho = \intt{N}_\rho$, for every \LCL-model $\mathcal{M}_\rho = \langle \mathcal{M}, \rho \rangle$.

If $M = N$ falls under axiom $(S)$, then there exist terms $P, Q, R$ such that the term $M$ is of the form  $\CS P Q R$ and the term $N$ is the term $(PR)(QR)$. We obtain
\begin{align*}\intt{\CS PQR}_\rho & = \left(\left(\intt{\CS}_\rho \cdot \intt{P}_\rho \right) \cdot \intt{Q}_\rho \right) \cdot \intt{R}_\rho \\ &  =  \left(\left(\cs \cdot \intt{P}_\rho \right) \cdot \intt{Q}_\rho \right) \cdot \intt{R}_\rho \\
 & = \left(\intt{P}_\rho \cdot \intt{R}_\rho \right) \cdot \left( \intt{Q}_\rho \cdot \intt{R}_\rho \right) \\
  & = \intt{\left(PR \right) \left(QR\right)}_\rho\end{align*}

Similarly, if $M = N$ falls under axiom $(K)$, then $M$ is an application  $\CK P Q$ for some terms $P, Q$ and $N$ is the term $P$, and we have
\[\intt{\CK P Q}_\rho = \left( \intt{\CK}_\rho \cdot \intt{P}_\rho \right) \cdot \intt{Q}_\rho = \left( \ck \cdot \intt{P}_\rho \right) \cdot \intt{Q}_\rho = \intt{P}_\rho \]

  If $M = N$ falls under axiom $(I)$, then the term $M$ is an application $\CI P$ and the term $N$ is a term $P$, and from the definition of the interpretation map and equation (\ref{eq:ci 2}) we obtain
\[\intt{\CI P}_\rho = \intt{\CI}_\rho \cdot \intt{P}_\rho = \ci \cdot \intt{P}_\rho  = \intt{P}_\rho \]

  If $M = N$ is obtained from $N = M$ by rule $\text{(sym)}$, then by the induction hypothesis we have $\intt{N}_\rho = \intt{M}_\rho$  for every \LCL-model $\mathcal{M}_\rho = \langle \mathcal{M}, \rho \rangle$, and the statement is a direct consequence of this.
  
 Next, suppose that $M = N$ is obtained by rule $\text{(trans)}$ applied to $M  = P$ and $P = N$. By the induction hypothesis we get $\intt{M}_\rho = \intt{P}_\rho$ and $\intt{P}_\rho = \intt{N}_\rho$, for every \LCL-model $\mathcal{M}_\rho = \langle \mathcal{M}, \rho \rangle$.  This implies $\intt{M}_\rho = \intt{N}_\rho$.
 
 Let $M = N$ be obtained by rule $\text{(app-l)}$ applied to $L = Q$. Then, terms $M$ and $N$ are of the form $LP$ and $QP$, respectively,  for some term $P$. By the induction hypothesis we have $\intt{L}_\rho = \intt{Q}_\rho$, for every \LCL-model $\mathcal{M}_\rho = \langle \mathcal{M}, \rho \rangle$. Then,
 \[\intt{LP}_\rho = \intt{L}_\rho \cdot \intt{P}_\rho = \intt{Q} _\rho \cdot \intt{P}_\rho = \intt{QP}_\rho.\]
 
 Similarly, if $M = N$ is obtained by rule $\text{(app-r)}$ applied to $L = Q$, then terms $M$ and $N$ are of the form $PL$ and $PQ$, respectively, for some term $P$. By the induction hypothesis we have $\intt{L}_\rho = \intt{Q}_\rho$, for every \LCL-model $\mathcal{M}_\rho = \langle \mathcal{M}, \rho \rangle$.  From the latter and the definition of the interpretation map we get
 \[\intt{PL}_\rho = \intt{P}_\rho \cdot \intt{L}_\rho = \intt{P}_\rho \cdot \intt{Q}_\rho = \intt{PQ}_\rho. \]
 
 Finally, we consider the case where $M = N$ is obtained by rule $\text{(ext)}$ from $Mx = Nx$ such that the variable $x$  appears neither in $M$ nor in $N$. Let $\mathcal{M}_\rho = \langle \mathcal{M}, \rho \rangle$ be an arbitrary model, where $\mathcal{M} = \langle D, \{A^\sigma\}_\sigma, \cdot, \cs, \ck, \ci \rangle$. The goal is to prove that $\intt{M}_\rho$ and $\intt{N}_\rho$ are equal. Let $d$ be an arbitrary element from the domain $D$.  
 By the induction hypothesis we have that  $\intt{Mx}_\rho = \intt{Nx}_\rho$ holds for any model $\mathcal{M}_\rho = \langle\mathcal{M}, \rho \rangle$. Let us consider the model $\langle \mathcal{M}, \rho(x := d)\rangle$.  Then, by the induction hypothesis we have $\intt{Mx}_{\rho(x : = d)} = \intt{Nx}_{\rho(x : = d)}$, which is equivalent to $\intt{M}_{\rho(x := d)}\cdot \intt{x}_{\rho(x : = d)} = \intt{N}_{\rho(x : =d)} \cdot \intt{x}_{\rho(x : = d)}$. Since $\intt{x}_{\rho(x : =d)} = \rho(x : = d) (x) = d $, it follows that $\intt{M}_{\rho(x : =d)} \cdot d = \intt{N}_{\rho(x : =d)}\cdot d$, for an arbitrary element $d$. On the other hand, the variable $x$ does not appear in terms $M$ and $N$, thus we have $\intt{M}_\rho = \intt{M}_{\rho(x := d)}$ and $\intt{N}_\rho = \intt{N}_{\rho(x := d)}$ by Lemma~\ref{lemma:diff env same meaning}. Therefore $\intt{M}_\rho \cdot d = \intt{N}_\rho \cdot d$, for every $d$ from the domain. From the extensionality of the operation $\cdot$ in \LCL-models we conclude $\intt{M}_\rho = \intt{N}_\rho$.
 
 This concludes the proof.
\end{proof}

\vspace{0.5em}
\subsection{Completeness}
\vspace{0.5em}

The converse to Soundness theorem is the so called Completeness theorem. It proves that whenever two terms $M, N$ have the same meaning in every \LCL-model, then they are equal, in the sense that the equation $M = N$ is provable in $\mathcal{EQ}^\eta$.

By $[M]$ we denote  the equivalence class of term $M$ with respect to the equivalence relation generated by the equational theory $\mathcal{EQ}^\eta$, \\ $\displaystyle [M] = \{N \mid M = N \text{ is provable in } \mathcal{EQ}^\eta\}$.

\begin{theorem}[Completeness of $\mathcal{EQ}^\eta$]\label{thm:completeness of eq theory}
If $\intt{M}_\rho = \intt{N}_\rho$ in every \LCL-model, then $M = N$ is provable in $\mathcal{EQ}^\eta$.
\end{theorem}
\begin{proof}
We construct an \LCL-model $\mathcal{M}_{\rho^\star} = \langle \mathcal{M}_0, \rho^\star \rangle$ where the applicative structure is defined as follows:
\[\mathcal{M}_0 = \langle D, \{ A^\sigma\}_\sigma, \cdot, \cs, \ck, \ci \rangle  \] 

\begin{itemize}
\item $D = \{[M] \mid M \in CL \}$;
\item $A^\sigma = \{[M] \mid  \ M \in CL \text{ and }\ \vdash M : \sigma \}$;
\item $[M] \cdot [N] = [MN]$;
\item $\cs = [\CS]$;
\item $\ck = [\CK]$;
\item $\ci = [\CI]$.
\end{itemize}
We define an environment $\rho^\star$ in the following way \[ \rho^\star(x) = [x].\]

In the definition of the set $A^\sigma$ the judgment $\vdash M  :\sigma$ means that $M : \sigma$ can be derived from the axioms and rules of Fig.~\ref{figure:axiomatization}, in other words, $M : \sigma$ is a theorem in \LCL. 

First we need to prove that this tuple is an applicative structure, so we prove that it satisfies the conditions of Definition~\ref{def:applicative structure}. The set $D$ is a non-empty set. Further, we have that for every $\sigma \in {\sf Types}$, the set $A^\sigma$ is a subset of $D$. The operation $\cdot$ defined by $[M] \cdot [N]$ is a binary operation on $D$, which is extensional.  The extensionality of the operation $\cdot$ is a consequence of the extensionality of combinatory logic. 
 Let $[M], [N] \in D$ be such that for all $[L] \in D$, $[M] \cdot [L] = [N] \cdot [L]$. If $x$ is a variable that does not appear neither in $M$ nor in $N$, then by hypothesis we have that $[M] \cdot [x] = [N] \cdot [x]$, which is equivalent to $[Mx] = [Nx]$. The latter implies that $Mx = Nx$ is provable in $\mathcal{EQ}^\eta$. Since $x$ is a variable that does not appear in $M$ and $N$,  from  $Mx = Nx$ we can prove $M = N$ by extensionality of $\mathcal{EQ}^\eta$. Thus $[M] = [N]$. Thus, we have proved that the operation $\cdot$ is extensional.
 Moreover, for $\sigma, \tau \in {\sf Types}$, let $[M] \in A^{\sigma \rightarrow \tau}$ and $[N] \in A^\sigma$. The former implies $\vdash M : \sigma  \rightarrow \tau $, whereas the latter implies $\vdash N  : \sigma $. From the assumptions, an instance of axiom scheme (Ax $4$), i.e. $(M : \sigma\rightarrow \tau) \Rightarrow ((N : \sigma) \Rightarrow ( MN : \tau))$, and by inference rule (MP), we obtain $\vdash MN : \tau$. It then follows that $[MN] \in A^\tau$. We have shown $\cdot : A^{\sigma \rightarrow \tau} \times A^\sigma \rightarrow A^\tau$.

Next, we prove that $[\CS]$ is an element of $D$ which satisfies equation~(\ref{eq:cs 1}) and equation~(\ref{eq:cs 2}) of Definition~\ref{def:applicative structure}. We have $\vdash \CS : (\sigma \rightarrow (\tau \rightarrow \rho)) \rightarrow (( \sigma \rightarrow \tau) \rightarrow (\sigma \rightarrow \rho))$ for every $\sigma, \rho, \tau \in {\sf Types}$, hence it  follows that $[\CS] \in A^{(\sigma \rightarrow (\tau \rightarrow \rho)) \rightarrow (( \sigma \rightarrow \tau) \rightarrow (\sigma \rightarrow \rho))}$, and the element $[\CS]$ satisfies the equation~(\ref{eq:cs 1}). Let $[M], [N], [L] \in D$. Then,
\begin{align*}
(([\CS] \cdot [M]) \cdot [N]) \cdot [L] & = [((\CS M)N)L] = [(ML)(NL)] \\ & = [ML] \cdot [NL] \\ 
& = ([M] \cdot [L]) \cdot ([N] \cdot [L]),
\end{align*}
so element $[\CS]$ satisfies the equation~(\ref{eq:cs 2}).

Similarly, $[\CK]$ is an element of $D$ which satisfies equations~(\ref{eq:ck 1}) and (\ref{eq:ck 2}). For every $\sigma, \tau \in {\sf Types}$, we have $\vdash \CK : \sigma \rightarrow (\tau \rightarrow \sigma)$, thus, $[\CK] \in A^{\sigma \rightarrow (\tau \rightarrow \sigma)}$ and element $[\CK]$ satisfies the equation~(\ref{eq:ck 1}). For $[M], [N] \in D$ we have
\begin{align*}
([\CK] \cdot [M]) \cdot [N] = [(\CK M)N] = [M].
\end{align*}
Hence, element $[\CK]$ satisfies  the equation~(\ref{eq:ck 2}) as well.

From $\vdash \CI : \sigma \rightarrow \sigma$, we get $[\CI] \in A^{\sigma \rightarrow \sigma}$ for every $\sigma\in {\sf Types}$. 
Further, for $[M]\in D$ it holds that 
\begin{align*}
[\CI] \cdot [M] = [\CI M] = [M]
\end{align*}
We have proved that the element $[\CI]$ of $D$ satisfies the equations~(\ref{eq:ci 1}) and (\ref{eq:ci 2}). 
Finally, we have proved that the tuple $\mathcal{M}_0$ is an applicative structure for \LCL. As a consequence, we have that $\mathcal{M}_{\rho^\star}
$ is an \LCL-model.

It is straightforward to prove that if an environment $\rho^\star$ is defined by $\rho^\star (x) = [x]$, then for every term $M$ we have $\intt{M}_{\rho^\star} = [M]$. 

If $\intt{M}_\rho = \intt{N}_\rho$ holds in every \LCL-model, then it also holds in model $\mathcal{M}_{\rho^\star}$ we have constructed. Thus, we have $\intt{M}_{\rho^\star} = \intt{N}_{\rho^\star}$. From the definition of environment $\rho^\star$ it follows that $[M] = [N]$, that is $M=N$ is provable in $\mathcal{EQ}^\eta$.
\end{proof}

Note that sets $A^\sigma$ 
 are not essential in this proof, since the theorem deals with the meaning of equal terms, without considering their types. In particular, we could choose any family $\{A^\sigma \}_\sigma$ of subsets of domain, which would satisfy conditions of Definition~\ref{def:applicative structure}.

\vspace{0.5em}
\subsection{Discussion}
\vspace{0.5em}

Soundness of the equational theory will be used in the proofs of soundness and completeness of the axiomatization for \LCL \ in the following Section~\ref{sec:semantical soundness and completeness}. If $M = N$ is provable in $\mathcal{EQ}^\eta$, then by Theorem~\ref{thm:soundness of eq theory} we have that $\intt{M}_\rho = \intt{N}_\rho$ in any model. As a consequence, the formula $\displaystyle M : \sigma \Rightarrow N : \sigma$, where $M = N$ is provable in the equational theory $\mathcal{EQ}^\eta$, is satisfied in any model. Since the formula $M : \sigma \Rightarrow N: \sigma$ is an instance of (Ax $5$), when $M = N$ is provable in the equational theory $\mathcal{EQ}^\eta$, we conclude that every instance of (Ax $5$) is satisfied in any model.

From the soundness and completeness of the equational theory $\mathcal{EQ}^\eta$ we have that $M = N$ is provable in $\mathcal{EQ}^\eta$ if and only if $\intt{M}_\rho = \intt{N}_\rho$ in every model. Furthermore, to prove that $M = N$ is provable in $\mathcal{EQ}^\eta$ it is enough to prove that $\intt{M}_{\rho^\star} = \intt{N}_{\rho^\star}$ in the model $\mathcal{M}_{\rho^\star}$, defined in the proof of Theorem~\ref{thm:completeness of eq theory}. However, both of these problems are undecidable. Since proving $M = N$ in $\mathcal{EQ}^\eta$ is not decidable, we have also that $N \in [M]$ is not decidable. As a consequence of this and the definition of term interpretation in environment $\rho^\star$, $\intt{M}_{\rho^\star}$, we have that $\intt{M}_{\rho^\star} = \intt{N}_{\rho^\star}$ is also not decidable.

 \vspace{0.5em}
\section{Soundness and completeness of the axiomatization for \LCL}\label{sec:semantical soundness and completeness}
\vspace{0.5em}

In this section, we prove that the logic  of combinatory logic \LCL\ is sound and complete w.r.t. the semantics introduced in \ref{sec:semantics}.  Hence this semantics is well suited to prove soundness and completeness for both the equational theory $\mathcal{EQ}^\eta$  and  the logic of combinatory logic  \LCL.

\subsection{Soundness of \LCL}
\vspace{0.5em}

%As usual, 
The property: whenever a formula is a deductive consequence of a set of formulas, it is a semantical consequence of that set, is referred to as the {\em soundness}.

\begin{theorem}[Soundness of \LCL]\label{thm:soundness}
If $T \vdash \alpha$, then $T \models \alpha$.
\end{theorem}
\begin{proof} First, we prove that every instance of each axiom scheme holds in every model and that the inference rule preserves the validity.

Let $\mathcal{M}_\rho$ be an \LCL-model. By Definition~\ref{def:meaning of term}, $\intt{\CS}_\rho = \cs$, $\intt{\CK}_\rho = \ck$, and $\intt{\CI}_\rho = \ci$. The element $\cs$ belongs to the set $A^{(\sigma \rightarrow  (\tau \rightarrow \rho)) \rightarrow ((\sigma \rightarrow \tau) \rightarrow (\sigma \rightarrow \rho))}$ for every $\sigma, \tau, \rho \in {\sf Types}$. Thus, we have that $\intt{\CS}_\rho \in A^{(\sigma \rightarrow ( \tau \rightarrow \rho)) \rightarrow ((\sigma \rightarrow \tau) \rightarrow (\sigma \rightarrow \rho))}$ and every instance of axiom scheme (Ax $1$) holds in every model.

Similarly, we have that $\ck \in A^{\sigma \rightarrow (\tau \rightarrow \sigma)}$ for every $\sigma, \tau \in {\sf Types}$. Therefore, we conclude that $\intt{\CK}_\rho \in A^{\sigma \rightarrow (\tau \rightarrow \sigma)}$ for every $\sigma, \tau \in {\sf Types}$, that is every instance of axiom scheme (Ax $2$) holds in every model.

Element $\ci$ belongs to set $A^{\sigma \rightarrow \sigma}$ for every $\sigma \in {\sf Types}$, thus every instance of axiom scheme (Ax $3$) holds in every model.

Next, we want to prove that every instance of (Ax $4$) is satisfied by  $\mathcal{M}_\rho$. Note that we consider the classical propositional connectives, that is $\alpha \Rightarrow \beta$ is not satisfied by $\mathcal{M}_\rho$ if and only if $\alpha$ is satisfied, and $\beta$ is not. Suppose that instance of (Ax $4$) is not satisfied by a model  $\mathcal{M}_\rho$, this is equivalent to $\mathcal{M}_\rho \models M : \sigma \rightarrow \tau$, $\mathcal{M}_\rho \models N : \sigma$ and $\mathcal{M}_\rho \not\models MN : \tau$. The hypothesis $\mathcal{M}_\rho \models M : \sigma \rightarrow \tau$ implies $\intt{M}_\rho \in A^{\sigma\rightarrow \tau}$ and the hypothesis $\mathcal{M}_\rho \models N : \sigma$ implies $\intt{N}_\rho \in A^{\sigma}$. By Definition~\ref{def:applicative structure} and Definition~\ref{def:meaning of term} we obtain that $\intt{MN}_\rho = \intt{M}_\rho \cdot \intt{N}_\rho \in A^{\tau}$. This is in contradiction with the assumption that $\mathcal{M}_\rho \not\models MN : \tau$. Thus, every instance of (Ax $4$) is satisfied by $\mathcal{M}_\rho$.

Further, we consider an instance of (Ax $5$). If $\mathcal{M}_\rho \models M : \sigma$, then $\intt{M}_\rho \in A^\sigma$, and if $M = N$ is provable in $\mathcal{EQ}^\eta$, then $\intt{M}_\rho = \intt{N}_\rho$ by Theorem~\ref{thm:soundness of eq theory}. From the latter and the assumption that $\intt{M}_\rho \in A^\sigma$, we obtain $\intt{N}_\rho \in A^\sigma$. Thus, $\mathcal{M}_\rho \models N : \sigma$,  and we conclude that every instance of (Ax $5$) is satisfied by $\mathcal{M}_\rho$.

It is not possible that a model does not satisfy an instance of (Ax $6$), since it would mean that    a model satisfies $\alpha$ and $\beta$, but does not satisfy $\alpha$ and it leads to a contradiction.

Similarly, for an instance of (Ax $7$). Let  $\mathcal{M}_\rho$ be a model such that:
\begin{enumerate}
\item[(1)] $\mathcal{M}_\rho \models \alpha \Rightarrow (\beta \Rightarrow \gamma)$, 
\item[(2)] $\mathcal{M}_\rho \models \alpha \Rightarrow \beta$,
\item[(3)] $\mathcal{M}_\rho \models \alpha$.
\end{enumerate} 
From $(2)$ and $(3)$ we obtain $\mathcal{M}_\rho \models \beta$. From the latter, $(1)$ and $(3)$ we obtain $\mathcal{M}_\rho \models \gamma$. Thus, every instance of (Ax $7$) is satisfied by a model $\mathcal{M}_\rho$.

Let us consider an instance of (Ax $8$). Suppose   that  it is not satisfied by a model $\mathcal{M}_\rho$, then \begin{enumerate}
\item[(1)] $\mathcal{M}_\rho \models \neg \neg \alpha \Rightarrow \neg \beta$,
\item[(2)] $\mathcal{M}_\rho \models \neg \neg \alpha \Rightarrow \beta$,
\item[(3)] $\mathcal{M}_\rho \not\models \neg \alpha$.
\end{enumerate} 
From $(3)$ we obtain $\mathcal{M}_\rho \models \neg \neg \alpha$. From the latter and $\mathcal{M}_\rho\models \neg \neg \alpha \Rightarrow \neg \beta$ we get $\mathcal{M}_\rho \models \neg \beta$. On the other hand, from $\mathcal{M}_\rho \models \neg \neg \alpha \Rightarrow \beta$ and $\mathcal{M}_\rho \models \neg \neg \alpha$ we get $\mathcal{M}_\rho \models \beta$. This is a contradiction with the previous conclusion that $\mathcal{M}_\rho \models \neg \beta$. Thus,  an instance of (Ax $8$) is satisfied by a model $\mathcal{M}_\rho$.

Next, we prove that the inference rule (MP) of Fig.~\ref{figure:axiomatization} preserves validity. If $\mathcal{M}_\rho \models\alpha \Rightarrow \beta$ and $\mathcal{M}_\rho \models \alpha$, then $\mathcal{M}_\rho \models \beta$ by Definition~\ref{def:satisfiability}. 
\end{proof}

\vspace{0.5em}
\subsection{Completeness of \LCL}
\vspace{0.5em}

The converse to Soundness theorem is called the  Completeness theorem. Completeness is the property: whenever a formula is semantical consequence of a set of formulas, it is also a deductive consequence of that set. We prove completeness by adapting the Henkin-style completeness method, which was developed for proving completeness of modal logic. The goal is to prove that every consistent set is satisfiable, then the completeness will be a straightforward consequence. % of it.

As we have already mentioned, one of the key points will be Deduction theorem (Theorem~\ref{thm:deduction}) we proved in 
\ref{sec:axiomatization}. Another important part of the proof is the extension of a consistent set to a maximal consistent set, which will be used in the construction of the canonical model.

We start with an auxiliary lemma which gives some useful properties of consistent sets.

\begin{lemma}\label{lemma:property of a consistent set} Let $T$ be a consistent set of formulas. For any formula $\alpha$, either $T \cup \{\alpha \}$ is consistent, or $T \cup \{\neg \alpha \}$ is consistent.

\end{lemma}
\begin{proof}
  Suppose the opposite, that both $T \cup \{\alpha \}$ and $T \cup \{ \neg \alpha\}$ are inconsistent, this is equivalent to $T , \alpha \vdash \bot$  and $T, \neg \alpha \vdash \bot$. By Theorem~\ref{thm:deduction} we obtain $T \vdash \alpha\Rightarrow \bot$ and $T \vdash \neg \alpha \Rightarrow \bot$. The former is equivalent to $T \vdash \neg \alpha$, and the latter is equivalent to $T \vdash \alpha$. 
    However, this contradicts  the assumption that $T$ is a consistent set. 
  Hence, it is not possible that both $T \cup \{\alpha \}$ and $T \cup \{ \neg \alpha\}$ are inconsistent.
\end{proof}

\begin{definition}\label{def:max consistent set} Let $T$ be a set of formulas. $T$ is a maximal consistent set if it is consistent and for any formula $\alpha$, either $\alpha \in T$ or $\neg \alpha \in T$.

\end{definition}

\begin{theorem}\label{thm:max consistent set} Every consistent set can be extended to a maximal consistent set.
\end{theorem}
\begin{proof} Let $T$ be a consistent set. We will construct a maximal consistent set $T^\star$ such that $T \subseteq T^\star$. Let $\alpha_0, \alpha_1, \ldots$ be an enumeration of all formulas from \LCL. We define a sequence of sets $T_i$, $i=0, 1, \ldots$ as follows:
\begin{enumerate}
\item[(1)] $T_0 = T$
\item[(2)] for every $i \geq 0$
\begin{enumerate}
\item if $T_i \cup \{\alpha_i \}$ is consistent, then $T_{i+1} = T_i \cup \{\alpha_i \}$, otherwise
\item $T_{i+1} = T_{i} \cup \{\neg \alpha_i\}$
\end{enumerate}
\item[(3)] $\displaystyle T^\star = \bigcup\limits_{i=0}^\infty T_i$
\end{enumerate}
We need to prove that $T^\star$ is a maximally consistent set. The construction of the set $T^\star$ ensures it is maximal. 

In order to prove that $T^\star$ is consistent, it is enough to prove that it is deductively closed and it does not contain all formulas. By construction of the sets $T_i$, it follows that every set $T_i$ is consistent.  
  Let $\alpha$ be an \LCL \ formula. We will prove that $\alpha \in T^\star$, whenever $T^\star \vdash \alpha$. If $\alpha = \alpha_j$ and $T_i \vdash \alpha$, then $\alpha \in T^\star$, since $T_{\text{max}\{i, j\} + 1}$ is consistent. Suppose that $T^\star \vdash \alpha$ and $\alpha_1, \alpha_2, \ldots, \alpha_n$ is a proof of $\alpha$ from $T^\star$. Since the sequence $\alpha_1, \alpha_2, \ldots, \alpha_n$ is finite, there exists a set $T_i$, such that $T_i \vdash \alpha$ and $\alpha\in T^\star$. 

Since $T^\star$ is a deductively closed set which does not contain all formulas, we conclude $T^\star$ is consistent.
\end{proof}

\begin{remark} Whenever we have $T \vdash \alpha$, $T$ is a consistent set, and $T^\star$ is a maximal consistent set, defined as in the proof  of Theorem~\ref{thm:max consistent set},  we have that $\alpha \in T^\star$. If $T$ is a consistent set and $T \vdash \alpha$, then $T \cup \{\alpha \}$ is consistent,  that is $T \cup \{\neg \alpha \}$ is inconsistent. Thus, $ \alpha \in T^\star$.

\end{remark}
We use a maximal consistent set $T^\star$, constructed in the proof of Theorem~\ref{thm:max consistent set}, to construct a canonical applicative structure.

\begin{definition}\label{def:canonical applicative strucutre} Let $T^\star$ be a maximal consistent set. A canonical applicative structure is a tuple $\mathcal{M}_{T^\star} = \langle D, \{A^\sigma \}_\sigma, \cdot, \cs, \ck, \ci\rangle$ where
\begin{itemize}
\item $D = \{[M] \mid M \in CL \}$,
\item $A^\sigma = \{[M] \mid M \in CL \text{ and } M : \sigma \in T^\star \}$,
\item $[M] \cdot [N] = [MN]$,
\item $\cs = [\CS]$,
\item $\ck = [\CK]$,
\item $\ci = [\CI]$.
\end{itemize}
\end{definition}

\begin{lemma}\label{lemma:canonical model} A canonical applicative structure $\mathcal{M}_{T^\star}$ is an applicative structure for \LCL.
\end{lemma}
\begin{proof}
We need to prove that the tuple $\mathcal{M}_{T^\star}$ defined in Definition~\ref{def:canonical applicative strucutre} satisfies conditions of Definition~\ref{def:applicative structure}. 

The set $D = \{[M] \mid M \in CL \}$ is a non-empty set.

For every $\sigma \in {\sf Types}$, the set $A^\sigma = \{[M] \mid M \in CL \text{ and } M : \sigma\in T^\star \}$ is a subset of the domain $D$. Since $M : \sigma \Rightarrow N : \sigma$ is an instance of axiom scheme whenever $M = N$ is provable in $\mathcal{EQ}^\eta$, we have that if $M = N$ is provable in $\mathcal{EQ}^\eta$, then $M : \sigma \Rightarrow N : \sigma \in T^\star$. Moreover, $T^\star$ is a maximal consistent set, and  for terms $M,N$ such that $M = N$ is provable in $\mathcal{EQ}^\eta$, it holds that $M : \sigma \in T^\star$ if and only if $N : \sigma \in T^\star$. As a consequence, the set $A^\sigma$ is well-defined.

The operation $\cdot$ defined by $[M] \cdot [N] = [MN]$ is an operation on $D$. The extensionality of the operation $\cdot$ is a consequence of the extensionality of combinatory logic. 
 We have proved extensionality of this operation in the proof of Theorem~\ref{thm:completeness of eq theory}, where the domain and the operation were defined exactly the same as here.
Moreover, for $\sigma, \tau \in {\sf Types}$, let $[M] \in A^{\sigma \rightarrow \tau}$ and $[N] \in A^\sigma$. The former implies $M : \sigma \rightarrow \tau \in T^\star$, while the latter implies $N  : \sigma \in T^\star$. Since $T^\star$ is deductively closed, we obtain $MN : \tau \in T^\star$, which is equivalent to $[MN] \in A^\tau$. We  have shown  that $\cdot : A^{\sigma \rightarrow \tau} \times A^\sigma \rightarrow A^\tau$.

Next, we prove that $[\CS]$ is an element of $D$ which satisfies equations~(\ref{eq:cs 1}) and (\ref{eq:cs 2}) of Definition~\ref{def:applicative structure}. From the fact that $\CS : (\sigma \rightarrow (\tau \rightarrow \rho)) \rightarrow (( \sigma \rightarrow \tau) \rightarrow (\sigma \rightarrow \rho))$ is an instance of axiom scheme (Ax $1$), we conclude that $\CS : (\sigma \rightarrow (\tau \rightarrow \rho)) \rightarrow (( \sigma \rightarrow \tau) \rightarrow (\sigma \rightarrow \rho)) \in T^\star$. It follows that $[\CS] \in A^{(\sigma \rightarrow (\tau \rightarrow \rho)) \rightarrow (( \sigma \rightarrow \tau) \rightarrow (\sigma \rightarrow \rho))}$. 
In the proof of Theorem~\ref{thm:completeness of eq theory} we have shown that the element $[\CS]$ satisfies the equation~(\ref{eq:cs 2}).

Similarly, $[\CK]$ is an element of $D$ which satisfies the equations~(\ref{eq:ck 1}) and (\ref{eq:ck 2}). For every $\sigma, \tau \in {\sf Types}$, $\CK : \sigma \rightarrow (\tau \rightarrow \sigma)$ is an instance of axiom scheme (Ax $2$) and it belongs to the set $T^\star$. Thus, $[\CK] \in A^{\sigma \rightarrow (\tau \rightarrow \sigma)}$. 
Proof that element $[\CK]$ satisfies equation~(\ref{eq:ck 2}) is given in the proof of Theorem~\ref{thm:completeness of eq theory}.

Since the formula $\CI : \sigma \rightarrow \sigma$ is an instance of axiom scheme (Ax $3$) for every $\sigma \in {\sf Types}$, it belongs to $T^\star$. Hence, $[\CI] \in A^{\sigma \rightarrow \sigma}$ for every $\sigma\in {\sf Types}$, and the element $[\CI]$ satisfies the equation~(\ref{eq:ci 1}). 
We have proved that the element $[\CI]$ satisfies the equation (\ref{eq:ci 2}) in the proof of Theorem~\ref{thm:completeness of eq theory}.

Thus, $\mathcal{M}_{T^\star}$ is an applicative structure for \LCL.
\end{proof}

Let $\mathcal{M}_{T^\star}$ be a canonical applicative structure. As we have already observed in the proof of Theorem~\ref{thm:completeness of eq theory}, if we define  an environment  $\rho^\star$ for $\mathcal{M}_{T^\star}$  by $\rho^\star (x)= [x]$, then for every term $M$ we have $\intt{M}_{\rho^\star} = [M]$.

\begin{definition} A canonical model is a tuple $\mathcal{M}_{T^\star, \rho^\star}=\langle \mathcal{M}_{T^\star}, \rho^\star \rangle$, where $\mathcal{M}_{T^\star}$ is a canonical applicative structure and $\rho^\star$ is an environment defined by $\rho^\star (x) = [x]$.

\end{definition}
\begin{lemma} Let $\mathcal{M}_{T^\star, \rho^\star}$ be a canonical model and $\alpha$ a formula. \[\mathcal{M}_{T^\star, \rho^\star} \models \alpha \text{ if and only if } \alpha \in T^\star. \]
\end{lemma}
\begin{proof}
We prove the statement by induction on the structure of the formula $\alpha$.

First, we consider the case where $\alpha$ is a statement $M : \sigma$. 
\begin{align*}
\ \mathcal{M}_{T^\star, \rho^\star} \models M : \sigma 
 \text{ if and only if } \  & \intt{M}_{\rho ^\star} \in A^\sigma \\
 \text{ if and only if }  \  & [M] \in A^\sigma %= \{[M] \mid M \in \LCL \text{ and } M : \sigma \in T^\star \}
 \\
  \text{ if and only if }  \  & M : \sigma \in T^\star. 
\end{align*}

Next, if $\alpha$ is a conjunction $\beta \wedge \gamma$ we have 
\begin{align*} %\hspace{-0.5cm}
 \mathcal{M}_{T^\star, \rho^\star} \models  \beta \wedge\gamma \  \text{ if and only if } \   &  \mathcal{M}_{T^\star, \rho^\star} \models \beta  \text{ and } \mathcal{M}_{T^\star, \rho^\star} \models \gamma \\
 \ \text{ if and only if } \ & \beta \in T^\star \text{ and } \gamma\in T^\star \\
\ \text{ if and only if } \ & \beta \wedge \gamma \in T^\star.
\end{align*}

Finally, let $\alpha$ be a negation $\neg \beta$. Then, 
\begin{align*}
\mathcal{M}_{T^\star, \rho^\star} \models \neg \beta & \text{ if and only if } \mathcal{M}_{T^\star, \rho^\star} \not\models \beta  \\
& \text{ if and only if }  \beta \not\in T^\star \\
& \text{ if and only if }  \neg \beta \in T^\star.
\end{align*}
\end{proof}

Finally, we can prove that for every consistent set there exists an \LCL-model which satisfies that set.

\begin{theorem}\label{thm:completeness} Every consistent set is satisfiable.
\end{theorem}
\begin{proof}
Let $T$ be a consistent set. Theorem~\ref{thm:max consistent set} ensures that the set $T$ can be extended to a maximal consistent set $T^\star$. For the canonical model $\mathcal{M}_{T^\star, \rho^\star}$  we know that $\mathcal{M}_{T^\star, \rho^\star}$ satisfies every formula from $T^\star$. Since $T \subseteq T^\star$, we have that every formula from $T$ is satisfied by $\mathcal{M}_{T^\star, \rho^\star}$. We have proved that there is a model which satisfies a consistent set $T$.
\end{proof}

\begin{theorem}[Completeness of \LCL]\label{thm:completeness2}
If $T \models \alpha$, then $T \vdash \alpha$.
\end{theorem}
\begin{proof}
If $T \models \alpha$, then $T \cup \{\neg \alpha \}$ is not satisfiable. By Theorem~\ref{thm:completeness} it follows that $T \cup \{\neg \alpha \} $ is inconsistent, i.e. $T \cup \{\neg \alpha \}  \vdash \bot$. The latter implies $T \vdash \alpha$.
\end{proof}

\begin{theorem}
$T \vdash \alpha$ if and only if  $T \models \alpha$ .
\end{theorem}

\vspace{0.5em}
\subsection{Application}
\vspace{0.5em}

An important consequence of Theorem~\ref{thm:soundness} is that combinatory logic is sound w.r.t. given semantics. Therefore the proposed semantics is a new semantics proven to be sound for combinatory logic.
\begin{theorem}[Soundness of $CL_\rightarrow$]\label{thm:soundness of CL}
If $\Gamma \vdash_{\TCL} M:\sigma$, then  $\Gamma \models  M:\sigma$ .
\end{theorem}
\begin{proof}
If $\Gamma \vdash_{\TCL} M : \sigma$, then there is a proof for $M : \sigma$ from $\Gamma$ in \LCL, by Proposition~\ref{proposition:from cl to extension}. Next, by Soundness theorem (Theorem~\ref{thm:soundness}) we obtain $\Gamma \models M : \sigma$.
\end{proof}

Nevertheless, the converse does not hold. For example, let us consider the following two statements: $x : \sigma$  and $\CK x y : \sigma$. The equation $\CK x y = x$ is provable in the equational theory $\mathcal{EQ}^\eta$, so the terms $\CK x y$ and $x$ have the same interpretation in every model, $\intt{\CK x y}_\rho = \intt{x}_\rho$, by Theorem~\ref{thm:soundness of eq theory}. As a result of this, we have that whenever $x : \sigma$ is satisfied by a model, then $\CK x y : \sigma$ is also satisfied by that model. Accordingly, the statement $\CK x y : \sigma$ is a semantical consequence of the statement  $x : \sigma$, and we write $x : \sigma \models \CK x y : \sigma$. If the simple type assignment system was  complete for the given semantics, then $x : \sigma \models \CK x y : \sigma$ would imply $x  :\sigma \vdash \CK x y : \sigma$, and the latter does not hold. We are not able to derive the statement $\CK x y : \sigma$ from the statement $x :  \sigma$ using axioms and rules of Fig.~\ref{fig: type assignment}.

To obtain a sound and complete type assignment system, %we can extend 
the system $CL_\rightarrow$ presented in Fig.~\ref{fig: type assignment} needs to be extended by the following additional rule:
\begin{align*}
\AxiomC{$\Gamma \vdash_{\TCL^=} M : \sigma$}
\AxiomC{$ M = N$ \text{ is provable in } $\mathcal{EQ}^\eta$}
\RightLabel{\text{(eq)}}
\BinaryInfC{$\Gamma \vdash_{\TCL^=} N : \sigma$}
\DisplayProof
\end{align*}

The type assignment system obtained by adding rule (eq) to the axioms and rules of Fig.~\ref{fig: type assignment} will be denoted by $CL_\rightarrow^=$, whereas derivability is denoted by $\vdash_{\TCL^=}$. The system $CL_\rightarrow^=$ is sound and complete w.r.t. the proposed semantics, as proved in the sequel.

\begin{theorem}[Soundness of $CL_\rightarrow^=$] If $\Gamma \vdash_{\TCL^=} M : \sigma$, then $\Gamma \models M : \sigma$.
\end{theorem}
\begin{proof}
Similarly to the proof of Theorem~\ref{thm:soundness of CL}, the soundness of $CL_\rightarrow^=$ can be obtained from the soundness of the axiomatization for \LCL. As the axiom scheme (Ax $5$) corresponds to the rule (eq),  Proposition~\ref{proposition:from cl to extension}  can also be proved for $CL_\rightarrow^=$: If $\Gamma \vdash 
_{\TCL^=} M : \sigma$, then  $\Gamma \vdash M : \sigma $ (in \LCL). From the soundness of the axiomatization for \LCL \ it follows that $\Gamma \vdash_{\TCL^=} M : \sigma$ implies $\Gamma \models M  :\sigma$.
\end{proof}

Moreover, the converse holds as well. Following the approach used in \cite{KasterovicG20}, we can define a  model $\mathcal{M}^\Gamma$ with respect to a basis $\Gamma$ such that $\mathcal{M}^\Gamma \models M  :\sigma$ if and only if $\Gamma \vdash_{\TCL^=}  M :\sigma$. Since the  model $\mathcal{M}^\Gamma$ is a model of the basis $\Gamma$, we can conclude that $\Gamma \models M : \sigma$ implies $\Gamma \vdash_{\TCL^=} M : \sigma$.  Hence, combinatory logic is sound and complete w.r.t. the proposed semantics.

\begin{definition}\label{def:canon app CL}
 Let $\Gamma$ be a basis.  We define a tuple $\mathcal{M}^{\Gamma}$ as follows. \[\mathcal{M}^{\Gamma} = \langle D, \{A^\sigma\}_\sigma, \cdot, \cs, \ck, \ci \rangle\] where
\begin{itemize}
\item $D = \{ [M] \mid M \in CL\}$;
\item $A^\sigma = \{[M] \mid M \in CL \text{ and } \Gamma \vdash_{\TCL^=} M : \sigma \}$;
\item $[M] \cdot [N] = [MN]$;
\item $\cs = [\CS]$;
\item $\ck = [\CK]$;
\item $\ci = [\CI]$;
\end{itemize}
\end{definition}
\begin{lemma}\label{lemma:canon app CL}
Let $\Gamma$ be a basis. The tuple $\mathcal{M}^{\Gamma}$ introduced in Definition~\ref{def:canon app CL} is an applicative structure.

\end{lemma}

\begin{proof}
Similarly to the proof of Theorem~\ref{thm:completeness of eq theory}, prove that a tuple  $\mathcal{M}^{\Gamma} = \langle D, \{A^\sigma\}_\sigma, \cdot, \cs, \ck, \ci  \rangle$ satisfies conditions of Definition~\ref{def:applicative structure}. The set $D$ is a non-empty set, and for every $\sigma \in {\sf Types}$, $A^\sigma \subseteq D$. The proof that a binary operation $\cdot$ on the set $D$ is extensional is already given in the proof of Theorem~\ref{thm:completeness of eq theory}. If $[M]  \in A^{\sigma
\rightarrow\tau}$ and $[N] \in A^{\sigma}$, then $\Gamma \vdash_{\TCL} M : \sigma\rightarrow\tau$ and $\Gamma \vdash_{\TCL} N : \sigma$. By the rule ($\rightarrow$ elim), we obtain $\Gamma \vdash MN : \tau$, and it follows that $[MN] \in A^\tau$. The proof that elements $[\CS]$, $[\CK]$ and $[\CI]$ satisfy equations \ref{eq:cs 2}, \ref{eq:ck 2} and \ref{eq:ci 2}, respectively, is given in the proof of Theorem~\ref{thm:completeness of eq theory}. For every $\sigma, \rho, \tau \in {\sf Types}$, we have $\Gamma \vdash_{\TCL} \CS : (\sigma \rightarrow (\rho \rightarrow \tau)) \rightarrow ((\sigma \rightarrow \rho) \rightarrow (\sigma \rightarrow \tau))$  by (axiom $K$) in Figure~\ref{fig: type assignment}. Thus, $\cs = [\CS] \in A^{(\sigma \rightarrow (\rho \rightarrow \tau)) \rightarrow ((\sigma \rightarrow \rho) \rightarrow (\sigma \rightarrow \tau))}$ for every  $\sigma, \rho, \tau \in {\sf Types}$ and the element $\cs$ satisfies condition \ref{eq:cs 1}. Similarly, we prove that elements $\ck$ and $\ci$ satisfy conditions \ref{eq:ck 1} and \ref{eq:ci 1}, respectively.
\end{proof}

\begin{lemma}\label{lemma:property of cl canon model} Let $\mathcal{M}^{\Gamma} = \langle D, \{A^\sigma\}_\sigma, \cdot, \cs, \ck, \ci  \rangle$ be an applicative structure defined in Definition~\ref{def:canon app CL} and environment $\rho^\star$ defined by $\rho^{\star} (x) = [x]$. Then, $\mathcal{M}^{\Gamma}_{\rho^\star} \models M : \sigma$ if and only if $\Gamma \vdash_{\TCL^=} M : \sigma$.
\end{lemma}
\begin{proof}
As we have already observed in the proof of Theorem~\ref{thm:completeness of eq theory}, it is straightforward to prove that $\intt{M}_{\rho^\star} = [M]$. Using this we obtain
\begin{align*}
\mathcal{M}^{\Gamma}_{\rho^\star} \models  M : \sigma
& \text{ if and only if } \intt{M}_{\rho^\star} \in A^\sigma \\ 
& \text{ if and only if } [M] \in \{[N] \mid N \in CL \text{ and } \Gamma \vdash_{\TCL^=} N : \sigma \} \\
& \text{ if and only if } \Gamma \vdash_{\TCL^=} M : \sigma.
\end{align*}
\end{proof}

\begin{theorem}[Completeness of $CL_\rightarrow^=$]\label{thm:completeness of CL}
Let $\Gamma$ be a basis. If  $\Gamma \models  M:\sigma$, then $\Gamma \vdash_{\TCL^=} M:\sigma$.
\end{theorem}
\begin{proof}
For a basis $\Gamma$, we can construct the model $\mathcal{M}^{\Gamma}_{\rho^\star}$ as in  Definition~\ref{def:canon app CL} and Lemma~\ref{lemma:property of cl canon model}. $\mathcal{M}^{\Gamma}_{\rho^\star}$ is a model of $\Gamma$, that is for every $x : \sigma \in \Gamma$ we have $\Gamma \vdash_{\TCL^=} x : \sigma$ and it follows that $\intt{x}_{\rho^\star} \in A^\sigma$, i.e. $\mathcal{M}^{\Gamma}_{\rho^\star} \models x : \sigma$. Thus, $\mathcal{M}^{\Gamma}_{\rho^\star} \models \Gamma$. From the later and the assumption    $\Gamma \models  M:\sigma$, we can conclude  
$\mathcal{M}^{\Gamma}_{\rho^\star} \models M: \sigma$. Using Lemma~\ref{lemma:property of cl canon model}  we obtain $\Gamma \vdash_{\TCL^=} M:\sigma$.
\end{proof}

The logic \LCL\ is a conservative extension of the simply typed combinatory logic. We have already stated in Proposition that whenever $M : \sigma$ can be typed in $\Gamma$ in the simply typed combinatory logic, we can derive formula $M:\sigma$ from $\Gamma$ in \LCL. Now, we will prove that the converse also holds.

\begin{theorem}\label{thm:conservative extension}
Let $\Gamma$ be a basis. If $\Gamma \vdash M : \sigma$, then $\Gamma \vdash_{\TCL} M : \sigma$.
\end{theorem}
\begin{proof}
Let $\Gamma$ be a basis and $\Gamma \vdash M : \sigma$. From Theorem~\ref{thm:completeness2} it follows  that $\Gamma \models M : \sigma$. Further, we get $\Gamma \vdash_{\TCL} M :\sigma$ by Theorem~\ref{thm:completeness of CL}.
\end{proof}

\vspace{0.5em}
\section{Conclusion}\label{sec:conclusion}
\vspace{0.5em}

We have introduced and studied the logic of combinatory  logic \LCL,  a new logical system designed as a propositional extension of simply typed combinatory logic. 
 It is  obtained by extending the language 
 of type assignment statements $M:\sigma$, $M$ being a combinatory term and $\sigma$ a simple type,  with classical propositional connectives, along with corresponding axioms and rules. On the other hand,  the logic \LCL\ can be seen as a logic obtained from classical propositional logic, where propositional letters are replaced by 
 type assignment statements  
$M : \sigma$. The logic \LCL\ captures the inference of type assignment statements in the simply typed combinatory logic.

 The usual approach in defining models for typed combinatory logic is to employ models of untyped combinatory logic and to interpret types as subsets of the untyped model. Given a valuation of term variables, an interpretation map, which interprets terms, is a function satisfying certain conditions, and given a valuation of type variables, the interpretation of types is defined inductively. In turn, in the model we propose the interpretation of term is defined inductively, and  interpretation of types is introduced as a family of sets  
 that  satisfy certain conditions. 
 The satisfiability of the type assignment statement depends on  
 whether the meaning of the term belongs to the corresponding set. This is further extended to the satisfiability of formulas containing propositional connectives in the usual way. The proof of soundness and completeness of the equational theory of untyped combinatory logic w.r.t. this semantics is presented. The given  axiomatization for the logic \LCL\ is proved sound and complete w.r.t. the proposed semantics. 
Consequently, this semantics is a new semantics for simply typed combinatory logic.

The proposed models are inspired by models introduced in \cite{MitchellM91, KasterovicG20}. In \cite{KasterovicG20} Kripke-style models for full simply typed $\lambda$-calculus and combinatory logic have been introduced as extensional applicative structures extended with special elements corresponding to primitive combinators provided with a valuation. We adapt this approach to simply typed combinatory logic and introduce models for \LCL.
The proposed semantics can be  extended to Kripke-style semantics, following the approaches used in \cite{MitchellM91} and \cite{KasterovicG20} for defining Kripke-style semantics for typed lambda calculus. In that case every world determines one \LCL -model, and if there is just one world we  essentially have the \LCL -model. Similarly, in Kripke-style semantics for classical propositional logic, if there is just one world, we have the truth-table semantics of classical propositional logic, \cite{N09}. Nevertheless, if we consider Kripke-style semantics, we will use the same canonical model as a model with one world, to prove completeness.

There are several compelling topics for further work. %We have envisaged  several possible direction for further 
We plan to further investigate combining typed calculi of combinatory logic and $\lambda$-calculus with other logics in two directions. One direction  is to enrich the language of logic and to study first-order extensions of simply typed calculi, \cite{Beeson04}. Combining simply typed combinatory logic and $\lambda$-calculus with classical first-order logic is a step forward towards the formalization of the meta-logic for simply typed theories. Moreover, the connection of these new systems with SMT solvers will further be deepened.
The other direction is to enrich the language of the underlying type theory and to study %the possibilities to extend 
 the extensions of intersection types and polymorphic types  with appropriate logics which enable formal reasoning about these type assignment systems. Furthermore,  we have envisaged to build on the  results of this paper and define formal models for probabilistic reasoning about simply typed terms of combinatory logic, \cite{GhilezanIKOS18}. Finally, we plan to investigate the application of the developed semantics in knowledge representation,  particularly in description logic based languages.

\bibliographystyle{plainurl}
\bibliography{references-KG}
\end{document}